\documentclass[twoside]{amsart}

\usepackage{amssymb}
\usepackage{amsfonts}
\usepackage{amsmath}
\usepackage{amsthm}
\usepackage{xcolor}
\usepackage{graphicx}
\usepackage{tikz}
\usepackage[margin={0.4in,0.8in}]{geometry}
\usepackage{tikz}
\usepackage{nopageno}
\usetikzlibrary{shapes.geometric}
\usetikzlibrary{quotes,angles}
\usepackage{float}

\usetikzlibrary{calc}

\definecolor{dark green}{rgb}{0,0.6,0}

\newtheorem{theorem}{Theorem}[section]
\newtheorem{lemma}[theorem]{Lemma}
\newtheorem{corollary}[theorem]{Corollary}

\theoremstyle{definition}
\newtheorem{definition}[theorem]{Definition}
\newtheorem{example}[theorem]{Example}

\theoremstyle{remark}

\numberwithin{equation}{section}

\title[Typed proof]{On the $P_3$-hull number and infecting times of generalized Petersen graphs}
\author[]{Daniel Herden, Jonathan Meddaugh, Mark Sepanski, Isaac Echols, Nina Garcia-Montoya, Cordell Hammon, Guanjie Huang, Adam Kraus, Jorge Marchena Menendez, Jasmin Mohn, Rafael Morales Jiménez}
\address{
All authors:
Department of Mathematics,
Baylor University,
Sid Richardson Building,
1410 S.4th Street,
Waco, TX 76706, USA}
\email{%daniel\_herden@baylor.edu, jonathan\_meddaugh@baylor.edu,
mark\_sepanski@baylor.edu}

\begin{document}
%%%%%% Header Creation

%\thanks{
%This work is a collaborative effort of the Baylor graduate course \emph{Research in Combinatorics}.}

\keywords{$P_3$-hull number, decycling number, infecting time, generalized Petersen graphs}
\subjclass[2020]{Primary: 05C05, 05C38, 05C85; Secondary 05C76}
\begin{abstract}
The $P_3$-hull number of a graph is the minimum cardinality of an infecting set of vertices that will eventually infect the entire graph under the rule that uninfected nodes become infected if two or more neighbors are infected. In this paper, we study the $P_3$-hull number
for generalized Petersen graphs and a number of closely related graphs that arise from surgery or more generalized permutations. In addition, the number of components of the complement of an infecting set of minimum cardinality is calculated for the generalized Petersen graph and shown to always be $1$ or $2$. Moreover, infecting times for infecting sets of minimum cardinality are studied. Bounds are provided and complete information is given in special cases.
\end{abstract}
\maketitle

%%%%%% Actual Paper
\section{Introduction}
As introduced in \cite{coffee-time} and \cite{geodetic}, the $P_3$-hull number of a simple connected graph is the minimum cardinality of a set $U$ of initially infected vertices that will eventually infect the entire graph where an uninfected node becomes infected if two or more of its neighbors are infected. There has been much work on formulas for the $P_3$-hull numbers of various types of graphs, \cite{p_3-Hamming, MR3040145, p_3-product, p3Kneser}, as well as with the closely related notion of the $2$-neighbor bootstrap percolation problem,
\cite{2-neighbour_bootstrap_percolation, Marcilon_Thiago, Przykucki_Michal}.

Important to this paper is the decycling number. Given a graph $G$, its decycling number, $\nabla(G)$, is the minimum cardinality of a set $U$ of vertices such that $G-U$ is acyclic. In general, it is very hard to compute a graph's decycling number. In fact, it has been shown to be NP-complete \cite{DecyclingNumberIsNPComplete}. However, results in special cases have been obtained, \cite{decycleOfGraphs, decylceRandom, DecyclingNumberOfGeneralizedPetersenGraphs, decycleCubic, decycleBoxProduct}.

In this paper, after initial definitions, we show that for a cubic graph, the $P_3$-hull number and the decycling number coincide, Theorem \ref{Main Theorem}. By \cite{DecyclingNumberOfGeneralizedPetersenGraphs}, it follows that the $P_3$-hull number of the generalized Petersen graph, $G(n,k)$, is $\left\lceil\frac{n+1}{2}\right\rceil$, Corollary \ref{p3_hull_GP}. Furthermore, the complement of any initial infecting set is a forest. In Theorem \ref{components thm}, it is show that for any infecting set of minimum cardinality this forest always has exactly one or two components.

In addition, we introduce the notion of the infecting time of an infecting set and study it for the Petersen graph. Explicit times are computed for special infecting sets, Theorem \ref{dan_time}. Giving explicit formulas for the minimal and maximal infecting times is a very difficult problem. However, we give complete answers for the special case of $G(n,1)$, Theorem \ref{full_time}.

Finally, we introduce a number of graphs related to the generalized Petersen graph. For a type of surgery, $G(n,k)\# G(n,k)$, the $P_3$-hull number is computed in Theorem \ref{thm_surgery}. Associated to a permutation $\sigma$ of $S_n$, a generalization of $G(n,k)$, called $GG(n,\sigma)$, is introduced. General bounds for its $P_3$-hull number are given in Theorem \ref{odd cycles upper bound}. An exact answer is computed under an additional hypothesis in Theorem \ref{thm:minGGP}.

\section{Initial Definitions}

    Throughout the paper, let $G = (V,E)$ be a finite, simple, connected graph and let $S\subseteq V$. We write $G[S]$ for the corresponding \textit{induced subgraph} of $G$ on $S$. We say $G$ is \textit{cubic} if each vertex of $G$ has degree $3$.

    Following \cite{p3Kneser}, the \textit{$P_3$-interval}, $I[S]$, is the set $S$ together with all vertices in $G$ that have two or more neighbors in $S$.
     If $I[S] = S$, then the set $S$ is called \textit{$P_3$-convex}.
     The $P_3$\textit{-convex hull}, $H_\mathcal C(S)$ of $S$, is the smallest $P_3$-convex set containing $S$.
     Iteratively, define $I^0[S] = S$ and $I^p[S] = I[I^{p-1}[S]]$ for any positive integer $p$. Then $H_\mathcal C(S)$ is the union of all $I^p[S]$.

     If $H_\mathcal C(S) = V$, we say that $S$ is a \textit{$P_3$-hull set} of $G$. We also refer to a $P_3$-hull set as an \textit{infecting set}. The minimum cardinality, $h_{P_3}(G)$, of a $P_3$-hull set in G is called the \textit{$P_3$-hull number} of G. This will be the main object of our study, and we will refer to $P_3$-hull sets of cardinality $h_{P_3}(G)$ as \textit{minimum size infecting sets}. For a $P_3$-hull set $S$, we say that the \textit{infecting time} of $S$, denoted $T_I(S)$, is the smallest integer $p$ such that $I^p[S] = V$.

    Relevant for this paper, we say that $S$ is a \textit{decycling set} of $G$ if the induced subgraph $G[V-S]$ is acyclic, \cite{decycleOfGraphs, decylceRandom}. The minimum cardinality of a decycling set of $G$ is called the \textit{decycling number} of $G$ and denoted by $\nabla(G)$.

Recall that a \textit{generalized Petersen graph}, $G(n,k)$ with $1\le k< \frac n2$, has vertex set
\[V=\{{u}_0,{u}_1,\ldots,{u}_{n-1},\,v_0,v_1,\dots,v_{n-1}\}\]
and edge set (interpreting each index modulo $n$)
\[E=\{{u}_i{u}_{i+1},\,{u}_iv_i, \,v_iv_{i+k}:0\leq i\leq n-1\}.\]

\section{$P_3$-Hull Numbers and Minimum Infecting Sets for $G(n,k)$}

\begin{theorem}\label{Main Theorem}
    Let $G = (V,E)$ be a cubic graph and $S\subseteq V$. Then $S$ is an infecting set of $G$ if and only if it is a decycling set of $G$. In particular, $h_{P_3}(G) = \nabla(G)$.
\end{theorem}

\begin{proof}
	We show first that an infecting set is a decycling set via the contrapositive. Assume that $S$ is not a decycling set. Then there exists some nonempty $W \subseteq V- S$ so that $G[W]$ is a cycle. Hence each $w\in W$ has exactly two neighbors in $W$ and one neighbor in $V- W$.
	
	Suppose now that $S$ is an infecting set. Since $S\cap W=\emptyset$, there is a minimal integer $n\ge 0$ so that $I^{n}[S] \cap W = \emptyset$ and $I^{n+1}[S] \cap W \not= \emptyset$. But then for $w\in I^{n+1}[S] \cap W$, there must have been two neighbors of $w$ in $I^{n}[S]$. As $I^{n}[S] \cap W = \emptyset$, these two neighbors must lie in $V-W$ which is a contradiction.
	
    Next we show that a decycling set is an infecting set.
    Suppose that $S$ is a decycling set of $G$. If $S$ is not infecting, then $H_\mathcal C(S)$ is a proper subset of $V$ that is still decycling. As the subgraph induced by $V-H_\mathcal C(S)$ is a forest, there exists some $v\in V-H_\mathcal C(S)$ with degree $1$. This implies two neighbors of $v$ lie in $H_\mathcal C(S)$ which is a contradiction.
\end{proof}

\begin{corollary} \label{p3_hull_GP}
$h_{P_3}(G(n,k)) = \lceil\frac{n+1}{2}\rceil$.
\end{corollary}

\begin{proof}
    This follows from Theorem \ref{Main Theorem} and  \cite[Theorem 3.1]{DecyclingNumberOfGeneralizedPetersenGraphs}, where it is shown that  $\nabla(G(n,k)) = \left\lceil\frac{n+1}{2}\right\rceil$.
\end{proof}

As a prelude to infecting time calculations, we present an explicit minimum size infecting set for $G(n,k)$ (examples can be seen in Figure \ref{ExamplesFigure}). An alternate example may be found in \cite[Lemma 3.3]{DecyclingNumberOfGeneralizedPetersenGraphs}.

\begin{corollary}\label{infecting_set}
	Let $c = \gcd(n,k)$,  $l=\frac{n}{c}$, and
	\[ \mathcal S_v = \{v_{j+ik}:0\leq i\leq l-1,\, 0\leq j\leq c-1,\, i\text{ odd }\}. \]
	For $l$ even, let
	\[ \mathcal S_u = \{u_0\}\]
	and for $l$ odd, let
	\[ \mathcal S_u = \{u_{c-1}\} \cup \{u_j:0\leq j\leq c-1,
	\, j\text{ even }\}. \]
	Then \[ \mathcal S = \mathcal S_v \cup \mathcal S_u \] is a minimum size infecting set for $G(n,k)$.
\end{corollary}

\begin{figure}[H]
	\centering
	\begin{tikzpicture}
		\node[fill,circle,inner sep=0pt,minimum size=1pt] (a) at (4*1,4*0) {};
		\node[fill,circle,inner sep=0pt,minimum size=1pt] (b) at (4*.866,.5*4) {};
		\node[fill,circle,inner sep=0pt,minimum size=1pt] (c) at (4*.5,.866*4) {};
		\node[fill,circle,inner sep=3pt,minimum size=1pt, label = \hspace{0em}$u_0$] (d) at (4*0,1*4) {};
		\node[fill,circle,inner sep=0pt,minimum size=1pt] (g) at (4*-1,0*4) {};
		\node[fill,circle,inner sep=0pt,minimum size=1pt] (f) at (4*-.866,.5*4) {};
		\node[fill,circle,inner sep=0pt,minimum size=1pt] (e) at (4*-.5,.866*4) {};
		\node[fill,circle,inner sep=0pt,minimum size=1pt] (j) at (4*0,-1*4) {};
		\node[fill,circle,inner sep=0pt,minimum size=1pt] (h) at (4*-.866,-.5*4) {};
		\node[fill,circle,inner sep=0pt,minimum size=1pt] (i) at (4*-.5,-.866*4) {};
		\node[fill,circle,inner sep=0pt,minimum size=1pt] (l) at (4*.866,-.5*4) {};
		\node[fill,circle,inner sep=0pt,minimum size=1pt] (k) at (4*.5,-.866*4) {};
		
		\draw[black] (a) -- (b);
		\draw[black] (b) -- (c);
		\draw[black] (c) -- (d);
		\draw[black] (d) -- (e);
		\draw[black] (e) -- (f);
		\draw[black] (f) -- (g);
		\draw[black] (g) -- (h);
		\draw[black] (h) -- (i);
		\draw[black] (i) -- (j);
		\draw[black] (j) -- (k);
		\draw[black] (k) -- (l);
		\draw[black] (l) -- (a);
		
		\node[fill,circle,inner sep=3pt,minimum size=1pt, label = \hspace{1.5em}$v_3$] (m) at (2.5*1,2.5*0) {};
		\node[fill,circle,inner sep=3pt,minimum size=1pt, label = \hspace{1.5em}$v_2$] (n) at (2.5*.866,.5*2.5) {};
		\node[fill,circle,inner sep=0pt,minimum size=1pt] (o) at (2.5*.5,.866*2.5) {};
		\node[fill,circle,inner sep=0pt,minimum size=1pt] (p) at (2.5*0,1*2.5) {};
		\node[fill,circle,inner sep=0pt,minimum size=1pt] (s) at (2.5*-1,0*2.5) {};
		\node[fill,circle,inner sep=3pt,minimum size=1pt, label = \hspace{1.5em}$v_{10}$] (r) at (2.5*-.866,.5*2.5) {};
		\node[fill,circle,inner sep=3pt,minimum size=1pt, label = \hspace{1.5em}$v_{11}$] (q) at (2.5*-.5,.866*2.5) {};
		\node[fill,circle,inner sep=3pt,minimum size=1pt, label = \hspace{1.5em}$v_6$] (v) at (2.5*0,-1*2.5) {};
		\node[fill,circle,inner sep=0pt,minimum size=1pt] (t) at (2.5*-.866,-.5*2.5) {};
		\node[fill,circle,inner sep=3pt,minimum size=1pt, label = \hspace{1.5em}$v_7$] (u) at (2.5*-.5,-.866*2.5) {};
		\node[fill,circle,inner sep=0pt,minimum size=1pt] (x) at (2.5*.866,-.5*2.5) {};
		\node[fill,circle,inner sep=0pt,minimum size=1pt] (w) at (2.5*.5,-.866*2.5) {};
		
		\draw[black] (n) -- (p);
		\draw[black] (p) -- (r);
		\draw[black] (r) -- (t);
		\draw[black] (t) -- (v);
		\draw[black] (v) -- (x);
		\draw[black] (x) -- (n);
		
		\draw[black, dashed] (m) -- (o);
		\draw[black, dashed] (o) -- (q);
		\draw[black, dashed] (q) -- (s);
		\draw[black, dashed] (s) -- (u);
		\draw[black, dashed] (u) -- (w);
		\draw[black, dashed] (w) -- (m);
		
		\draw[black] (m) -- (a);
		\draw[black] (n) -- (b);
		\draw[black] (o) -- (c);
		\draw[black] (p) -- (d);
		\draw[black] (q) -- (e);
		\draw[black] (r) -- (f);
		\draw[black] (s) -- (g);
		\draw[black] (t) -- (h);
		\draw[black] (u) -- (i);
		\draw[black] (v) -- (j);
		\draw[black] (w) -- (k);
		\draw[black] (x) -- (l);
		
	\end{tikzpicture}
	\qquad
	\begin{tikzpicture}
		\node[fill,circle,inner sep=3pt,minimum size=1pt, label = \hspace{1.5em}$u_3$] (a) at (4*1,4*0) {};
		\node[fill,circle,inner sep=3pt,minimum size=1pt, label = \hspace{1.5em}$u_2$] (b) at (4*.866,.5*4) {};
		\node[fill,circle,inner sep=0pt,minimum size=1pt] (c) at (4*.5,.866*4) {};
		\node[fill,circle,inner sep=3pt,minimum size=1pt, label = \hspace{0em}$u_0$] (d) at (4*0,1*4) {};
		\node[fill,circle,inner sep=0pt,minimum size=1pt] (g) at (4*-1,0*4) {};
		\node[fill,circle,inner sep=0pt,minimum size=1pt] (f) at (4*-.866,.5*4) {};
		\node[fill,circle,inner sep=0pt,minimum size=1pt] (e) at (4*-.5,.866*4) {};
		\node[fill,circle,inner sep=0pt,minimum size=1pt] (j) at (4*0,-1*4) {};
		\node[fill,circle,inner sep=0pt,minimum size=1pt] (h) at (4*-.866,-.5*4) {};
		\node[fill,circle,inner sep=0pt,minimum size=1pt] (i) at (4*-.5,-.866*4) {};
		\node[fill,circle,inner sep=0pt,minimum size=1pt] (l) at (4*.866,-.5*4) {};
		\node[fill,circle,inner sep=0pt,minimum size=1pt] (k) at (4*.5,-.866*4) {};
		
		\draw[black] (a) -- (b);
		\draw[black] (b) -- (c);
		\draw[black] (c) -- (d);
		\draw[black] (d) -- (e);
		\draw[black] (e) -- (f);
		\draw[black] (f) -- (g);
		\draw[black] (g) -- (h);
		\draw[black] (h) -- (i);
		\draw[black] (i) -- (j);
		\draw[black] (j) -- (k);
		\draw[black] (k) -- (l);
		\draw[black] (l) -- (a);
		
		\node[fill,circle,inner sep=0pt,minimum size=1pt] (m) at (2.5*1,2.5*0) {};
		\node[fill,circle,inner sep=0pt,minimum size=1pt] (n) at (2.5*.866,.5*2.5) {};
		\node[fill,circle,inner sep=0pt,minimum size=1pt] (o) at (2.5*.5,.866*2.5) {};
		\node[fill,circle,inner sep=0pt,minimum size=1pt] (p) at (2.5*0,1*2.5) {};
		\node[fill,circle,inner sep=0pt,minimum size=1pt] (s) at (2.5*-1,0*2.5) {};
		\node[fill,circle,inner sep=0pt,minimum size=1pt] (r) at (2.5*-.866,.5*2.5) {};
		\node[fill,circle,inner sep=0pt,minimum size=1pt] (q) at (2.5*-.5,.866*2.5) {};
		\node[fill,circle,inner sep=3pt,minimum size=1pt, label = \hspace{1.5em}$v_6$] (v) at (2.5*0,-1*2.5) {};
		\node[fill,circle,inner sep=0pt,minimum size=1pt] (t) at (2.5*-.866,-.5*2.5) {};
		\node[fill,circle,inner sep=3pt,minimum size=1pt, label = \hspace{1.5em}$v_7$] (u) at (2.5*-.5,-.866*2.5) {};
		\node[fill,circle,inner sep=3pt,minimum size=1pt, label = \hspace{1.5em}$v_4$] (x) at (2.5*.866,-.5*2.5) {};
		\node[fill,circle,inner sep=3pt,minimum size=1pt, label = \hspace{1.5em}$v_5$] (w) at (2.5*.5,-.866*2.5) {};

		\draw[black] (p) -- (t);
		\draw[black] (t) -- (x);
		\draw[black] (x) -- (p);
		
		\draw[black, dashed] (q) -- (u);
		\draw[black, dashed] (u) -- (m);
		\draw[black, dashed] (m) -- (q);
		
		\draw[black, densely dotted] (r) -- (v);
		\draw[black, densely dotted] (v) -- (n);
		\draw[black, densely dotted] (n) -- (r);
		
		\draw[black, dashdotted] (s) -- (w);
		\draw[black, dashdotted] (w) -- (o);
		\draw[black, dashdotted] (o) -- (s);
		
		\draw[black] (m) -- (a);
		\draw[black] (n) -- (b);
		\draw[black] (o) -- (c);
		\draw[black] (p) -- (d);
		\draw[black] (q) -- (e);
		\draw[black] (r) -- (f);
		\draw[black] (s) -- (g);
		\draw[black] (t) -- (h);
		\draw[black] (u) -- (i);
		\draw[black] (v) -- (j);
		\draw[black] (w) -- (k);
		\draw[black] (x) -- (l);
		
	\end{tikzpicture}
	\centering
	\caption{Minimum size infecting sets for $G(12,2)$ and $G(12, 4)$, respectively.}
	\label{ExamplesFigure}
\end{figure}
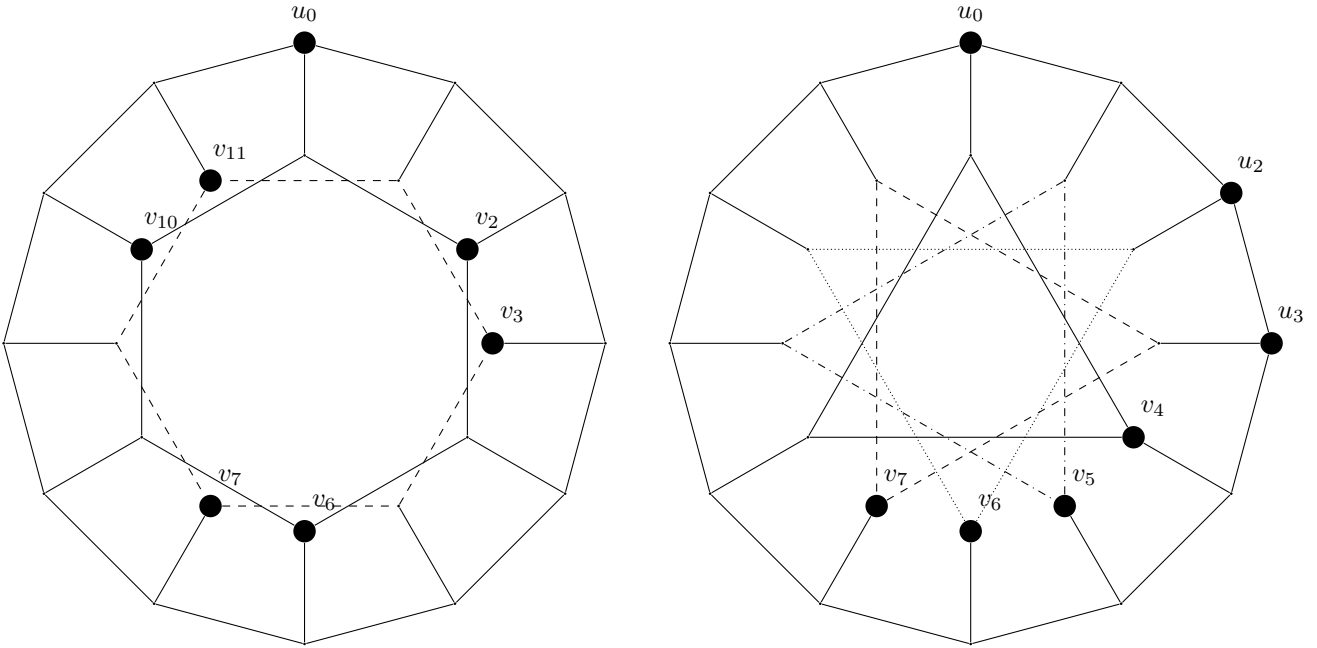

\begin{proof}
	Note that the subgraph of $G(n,k)$ which is induced on the vertex set $\{v_0,v_1,\ldots,v_{n-1}\}$ is the disjoint union of $c$ cycles of length $l$. The corresponding vertex set of each cycle is given by 	\[ V_j = \{v_{j+ik}:0\leq i\leq l-1\}\]
	with $0\leq j \leq c-1$.
	
	We will distinguish two cases. The first is when $l$ is even.
	Given $\mathcal S$ as an initial set of infected points, the infection will spread to infect every vertex set $V_j$, $0\leq j\leq c-1$. From there with ${u}_0$ the infection spreads to all of $\{{u}_0, {u}_1, \ldots, {u}_{n-1}\}.$ Note that $n=cl$ is even and
	\[|\mathcal S|=1+c\frac{l}{2}=\frac{n}{2}+1=\bigg\lceil\frac{n+1}{2}\bigg\rceil\]
	so that $\mathcal S$ is a minimum size infecting set by Corollary \ref{p3_hull_GP}.
	
	Turn now to the case of $l$ odd.
	Given $\mathcal S$ as an initial set of infected points, the infection will spread to $\{{u}_j:0\leq j\leq c-1\}.$ With ${u}_j$ and $v_{j+k}$ infected, also $v_j$ will be infected, and the infection will spread to infect every vertex set $V_j,\hspace{1mm} 0\leq j\leq c-1$. From there with $\{{u}_j:0\leq j\leq c-1\}$ the infection spreads to all of $\{{u}_0,{u}_1,\ldots,{u}_{n-1}\}.$ As
	\[|\mathcal S|=\bigg\lceil\frac{c+1}{2}\bigg\rceil +c\frac{l-1}{2}=\bigg\lceil c\frac{l-1}{2}+\frac{c+1}{2}\bigg\rceil=\bigg\lceil\frac{n+1}{2}\bigg\rceil,\]
	it follows that $\mathcal S$ is a minimum size infecting set.
\end{proof}

By Theorem \ref{Main Theorem}, the complement of a minimum size infecting set of a cubic graph is a forest. The next theorem constrains the number of connected components of this forest for $G(n,k)$.

\begin{theorem}\label{components thm}
	Let $S$ be a minimum size infecting set of the generalized Petersen graph $G=G(n,k)=(V,E)$.

\begin{itemize}	
	\item If $n$ is odd, then $G[V-S]$ is a tree.
	
	\item If $n$ is even, then the forest $G[V-S]$ may have two or one connected components. It will have two connected components if and only if $S$ has no neighboring points and one connected component if and only if $S$ has exactly one pair of points that are neighbors.
\end{itemize}
\end{theorem}

\begin{proof}
	Recall for $G$ that $|V|=2n$ and $|E|=3n$. Write $\nu$ and $\epsilon$ for the number of vertices and edges, respectively, of $G[V-S]$. As $G[V-S]$ is a forest, $\nu-\epsilon$ is the number of trees in the forest.
	
	When $n$ is odd, write $n=2k+1$ so that $|S|=\lceil \frac{n+1}{2} \rceil= k+1$ and $\nu=2(2k+1) - (k+1) = 3 k +1$. At most, if $S$ has no neighboring points, passing from $G$ to $G[V-S]$ would remove $3|S|$ edges. Therefore, $\epsilon
	\geq 3(2k+1) - 3(k+1) = 3k$. It follows that $\nu- \epsilon \le 1$, and $G[V-S]$ must be a single tree.
	
	When $n$ is even, write $n=2k$ so $|S|=\lceil \frac{n+1}{2} \rceil= k+1$ and $\nu = 2(2k) - (k+1) = 3 k -1$. As in the previous paragraph, we get $\epsilon
	\geq 3(2k) - 3(k+1) = 3k-3$. It follows that $\nu- \epsilon \le 2$, and $G[V-S]$ has either one or two connected components. In addition, $G[V-S]$ is a tree exactly when $S$ has no neighboring points, and $G[V-S]$ has two trees exactly when $S$ has exactly one pair of neighboring points.
\end{proof}

\section{Infecting Times}

\begin{theorem}\label{dan_time}
	Let $c = \gcd(n,k)$,  $l=\frac{n}{c}$, and $\mathcal S$ be the infecting set for $G(n,k)$ from Corollary \ref{infecting_set}.
\begin{itemize}	
	\item When $l$ is even, the infecting time for $\mathcal S$ is $\frac{n}{2}$.
    \item When $l$ is odd, the infecting time for $\mathcal S$ is $\frac{n-c}{2}+1$.
\end{itemize}
\end{theorem}

\begin{proof}
Note that $n = cl$ with $1\le c \le k < \frac n2$. Thus $l \ge 3$.

Begin with the case of $l\ge 4$ even. Obviously, $\{u_0\} \cup \{v_k : 0\leq k \leq n-1\}\subseteq I^1[\mathcal S]$ by definition. However, also $u_{n-1} \in I^1[\mathcal S]$. To see this, note that $v_{n-1} \in V_{c-1}$, and $(c-1)+ik \equiv n-1 \mbox{ mod } n$ holds iff $i\frac kc \equiv -1 \mbox{ mod } l$. As $l$ is even, the last congruency can only hold for $i$ odd. Thus, $v_{n-1} \in \mathcal S_v \subseteq \mathcal S$ and $u_{n-1} \in I^1[\mathcal S]$.

In addition to $u_{n-1} \in I^1[\mathcal S]$, $u_{1} \in I^1[\mathcal S]$ happens if and only if $c=1$. In any case, another $\frac n2 -1$ steps are necessary for the infection to spread from $I^1[\mathcal S]$ to the rest of the graph.

	%Begin with the case of $l$ even. Note that $n$ is even in this case. By construction, at time $t=1$, the infected set will be $\{u_0\} \cup \{v_k : 0\leq k \leq n-1\}$. From there, $\{u_{n-1}, u_1\}$ will be added at time $t=2$, $\{u_{n-2}, u_2\}$ at $t=3$, and, in general, $\{u_{n-m+1}, u_{m-1}\}$ at $t=m$. This will exhaust all vertices at time $t=\frac{n}{2}+1$.

	%Turn now to the case of $l$ odd. Note that $n$ will have the same parity as $c$ in this case. By construction, at time $t=3$, the infected set will be $\{u_0,\ldots,u_{c-1}\} \cup \{u_{n-2}, u_{n-1}, u_c, u_{c+1}\}\cup \{v_k : 0\leq k \leq n-1\}$. At time $t=m$, $\{u_{n-m+1}, u_{c+m-2}\}$ will be added. This will exhaust all vertices at time $t=\frac{n-c}{2}+1$.

Turn now to the case of $l \ge 3$ odd. Obviously, \[\{u_j: 0\le j\le c-1\} \cup \{v_{j+ik} : 0< i < l-1, 0\le j\le c-1\}\subseteq I^1[\mathcal S]\] by definition, where some of the $v_{j+ik}$ for $i\in \{0,l-1\}$ will still be missing. In addition, exactly one of $u_c$ and $u_{n-1}$ will be in $I^1[\mathcal S]$. To see this, note that $v_{c} \in V_{0}$, and $i_1k \equiv c \mbox{ mod } n$ holds iff $i_1\frac kc \equiv 1 \mbox{ mod } l$. Similarly,
$v_{n-1} \in V_{c-1}$, and $(c-1)+i_2k \equiv n-1 \mbox{ mod } n$ holds iff $i_2\frac kc \equiv -1 \mbox{ mod } l$. Thus, $i_1\frac kc + i_2\frac kc\equiv 0 \mbox{ mod } l$, and $i_1+i_2$ is a multiple of $l$. With $0<i_1,i_2< l$, we must have $i_1+i_2 = l$ odd. Thus, exactly one of $i_1,i_2$ is odd, exactly one of $v_c, v_{n-1}$ is in $\mathcal S_v \subseteq \mathcal S$, and exactly one of $u_c,u_{n-1}$ is in $I^1[\mathcal S]$. In particular, \[\mbox{$I^1[\mathcal S]$ contains exactly $c+1$ of the points $u_j$. These $c+1$ points are the vertices of a path of length $c$ on the exterior cycle.}\]

Continuing with $v_c$ and $v_{n-1}$, we have $v_c \in I^1[\mathcal S]$ and $u_c \in I^2[\mathcal S]$. To see this, note that $v_c \in I^1[\mathcal S]$ fails iff $i_1 = l-1$. However, $i_1 = l-1$ implies $\frac kc \equiv -1 \mbox{ mod } l$ and $k+c \equiv 0 \mbox{ mod } n$ which is impossible for $1\le c \le k < \frac n2$. A similar argument shows that $u_{n-1} \in I^2[\mathcal S]$ unless $k=c =\gcd(n,k)$, i.e., unless $n$ is a multiple of $k$. Thus, $u_c, u_{n-1} \in I^2[\mathcal S]$ for $k \nmid n$ while a direct inspection shows $u_c, u_{c+1} \in I^2[\mathcal S]$ for $k \mid n$. In any case,
\[\mbox{$I^2[\mathcal S]$ will add at least one new point $u_j$ to $I^1[\mathcal S]$ but not more than two.}\]

For $c\in \{1,2\}$, note $\{v_k : 0\leq k \leq n-1\}\subseteq I^2[\mathcal S]$. If $c=1$, then $n=l$ is odd and the infection will need another $\frac {n+1}2 -2$ steps to spread from $I^2[\mathcal S]$ to the rest of the graph.
If $c=2$, then $n=2l$ is even and the infection will need another $\frac {n}2 -2$ steps to spread from $I^2[\mathcal S]$ to the rest of the graph.

For $c \ge 3$, we have $\{v_k : 0\leq k \leq n-1\}\subseteq I^3[\mathcal S]$. If $k\nmid n$ and $u_c \in I^1[\mathcal S]$, then $u_{c+1}, u_{n-1} \in I^2[\mathcal S]$, $u_{c+2}, u_{n-2} \in I^3[\mathcal S]$, and the infection will need a total of $\frac {n-c}2 +1$ steps to spread from $\mathcal S$ to the rest of the graph. A similar argument is true for $k\nmid n$ and $u_{n-1} \in I^1[\mathcal S]$.
If $k\mid n$, then $u_{c+1} \in I^2[\mathcal S]$, $u_{c+2}, u_{n-1} \in I^3[\mathcal S]$, and the infection will need again a total of $\frac {n-c}2 +1$ steps to spread from $\mathcal S$ to the rest of the graph.
\end{proof}

\begin{lemma}\label{diam_lemma}
	For a cubic graph, $G=(V,E)$, and infecting set $S$, the infecting time is $\lceil \frac{d+1}{2}\rceil$ where $d$ is the largest diameter of a tree of the forest $G[V-S]$.
\end{lemma}

\begin{proof}
	From Theorem \ref{Main Theorem}, $G[V-S]$ is a forest. The result follows by observing that since $G$ is cubic, the infection will first infect the leaves of $G[V-S]$ and progress in a similar fashion iteratively.
\end{proof}

In general, determining the minimal or maximal infecting time of a minimum size infecting set for $G(n,k)$ is very difficult. We give next full results for the special case of $G(n,1)$, including an exhaustive list of all possible infecting times. By Lemma \ref{diam_lemma}, it suffices to calculate the maximal diameter for the connected components of the complement of a minimum size infecting set. We give these maximal diameters below as it is finer information.

\begin{theorem}\label{full_time}
	For $G(n,1)$, the set of possible maximal diameters of connected components for the complement of a minimum size infecting set $S$ is
	\[ \bigg\{n, n+2, n+4, \ldots, n +2 \bigg\lfloor\frac{n-3}{4}\bigg\rfloor \bigg\} \]
	when $n\ge 3$ is odd and
	\[ \bigg\{d : \bigg\lfloor \frac{n}{2} \bigg\rfloor \leq d \leq \frac 32 n -2 \text{ with }d \not= n-1 \bigg\} \]
	when $n\ge 4$ is even.
\end{theorem}

\begin{proof}
	
	Begin by viewing $G=G(n,1)$ as a discrete annulus as in Figure \ref{fig:ladder}.
		
	\begin{figure}[H]
		\centering
		\begin{tikzpicture}
			\tikzstyle{hollow node}=[draw,circle, fill=white, outer sep=-4pt,inner sep=0pt,
			minimum width=8pt, above]
			\foreach \i in {0,...,5,7,8}{
				\foreach \j in {0,1}{
					\draw[black,thick] (\i,\j) -- (\i+1,\j);
				}
			}
			\foreach \i in {0,...,9}{
				\draw[black,thick] (\i,0) -- (\i,1);
			}
			\draw[black,thick] (6,0) -- (6.25,0);
			\draw[black,thick] (6,1) -- (6.25,1);
			\draw[black,thick] (6.75,0) -- (7,0);
			\draw[black,thick] (6.75,1) -- (7,1);
			\draw[black,thick] (-.4,0) -- (0,0);
			\draw[black,thick] (-.4,1) -- (0,1);
			\draw[black,thick] (9,1) -- (9.4,1);
			\draw[black,thick] (9,0) -- (9.4,0);
			
			\foreach \i in {0,1,2}{
				\foreach \j in {0,1}{
					\draw  node[fill,circle,inner sep=0pt,minimum size=2pt] at (6.36 + .14*\i,\j){};
					\draw  node[fill,circle,inner sep=0pt,minimum size=2pt] at (9.6 + .14*\i,\j){};
					\draw  node[fill,circle,inner sep=0pt,minimum size=2pt] at (-.6 - .14*\i,\j){};
				}
			}
			\foreach \i in {0,...,9}{
				\node[hollow node] (\i0) at (\i,0){};
				\node[hollow node] (\i1) at (\i,1){};
			}
			\foreach \i in {0,...,6}{
				\node at (\i,-.4){$v_{\i}$};
				\node at (\i,1.4){$u_{\i}$};
			}
			\foreach \i in {1,2,3}{
				\node at (10-\i,-.4){$v_{n-\i}$};
				\node at (10-\i,1.4){$u_{n-\i}$};
			}
		\end{tikzpicture}
		\caption{}
		\label{fig:ladder}
	\end{figure}
	
	Start with the case where $n$ is odd ($n\geq 3$) and write $n=2k+1$.
	By Corollary \ref{p3_hull_GP} and Theorem \ref{components thm}, every minimum size infecting set, $S$, contains $k+1$ vertices, the remaining vertices constitute a tree, $T$, and no two vertices of $S$ are neighbors. To prevent a cycle from appearing in $T$, there must exist a (square) face of the annulus whose top and bottom edges both are not in $T$. From this it follows, possibly after turning the annulus over, that there is a face with exactly two vertices of $S$ configured as in Figure~\ref{fig:1}, where marked points belong to $S$.

\begin{figure}[H]
	\centering
	\begin{tikzpicture}
		\tikzstyle{hollow node}=[draw,circle, fill=white, outer sep=-4pt,inner sep=0pt,
		minimum width=8pt, above]
		\foreach \i in {0,...,0}{
			\foreach \j in {0,1}{
				\draw[black,thick] (\i,\j) -- (\i+1,\j);
			}
		}
		\foreach \i in {0,...,1}{
			\draw[black,thick] (\i,0) -- (\i,1);
		}
		\draw[black,thick] (-.4,0) -- (0,0);
		\draw[black,thick] (-.4,1) -- (0,1);
	
	\draw[black,thick] (1.4,0) -- (1,0);
	\draw[black,thick] (1.4,1) -- (1,1);
		
		\foreach \i in {0,1,2}{
			\foreach \j in {0,1}{
				\draw  node[fill,circle,inner sep=0pt,minimum size=2pt] at (1.56 + .14*\i,\j){};
			%	\draw  node[fill,circle,inner sep=0pt,minimum size=2pt] at (9.6 + .14*\i,\j){};
				\draw  node[fill,circle,inner sep=0pt,minimum size=2pt] at (-.6 - .14*\i,\j){};
			}
		}
		\foreach \i in {0,1}{
			\node[hollow node] (\i0) at (\i,0){};
			\node[hollow node] (\i1) at (\i,1){};
		}

	\tikzstyle{red node}=[draw,circle, fill=red, outer sep=-4pt,inner sep=0pt, minimum width=8pt, above]
	
	\node[red node] at (0,0){};
	\node[red node] at (1,1){};

	\end{tikzpicture}
	\caption{}
	\label{fig:1}
\end{figure}

	After relabeling, the top left vertex can be labeled as $u_{n-1}$ and the bottom right as $v_0$. Cut the annulus now through the middle of this square and unfold it. The tree $T$ must connect $v_0$ to $u_{2k}$ and is forced to contain $u_1,v_1,u_{2k-1},$ and $v_{2k-1}$, see Figure \ref{fig:ladder}. Recalling that $S$ has no neighboring vertices, to prevent $T$ from having a cycle around the face with indices $1$ and $2$, $S$ must contain exactly one element of $\{u_2,v_2\}$. Arguing iteratively, it follows that each set $\{u_{2i}, v_{2i}\}$, $0\leq i \leq k$ contains exactly one element of $S$.
	
	The shortest diameter for such a tree $T$ is pictured in Figure \ref{fig:narrowest odd tree}. It has diameter $n$ and, by Lemma \ref{diam_lemma}, a corresponding infecting time of $\lceil \frac{n+1}{2} \rceil$.

	\begin{figure}[H]
		\centering
		\begin{tikzpicture}
			\tikzstyle{hollow node}=[draw,circle, fill=white, outer sep=-4pt,inner sep=0pt,
			minimum width=8pt, above]
			\tikzstyle{red node}=[draw,circle, fill=red, outer sep=-4pt,inner sep=0pt,
			minimum width=8pt, above]
			\foreach \i in {0,...,5,7,8,9}{
				\foreach \j in {0,1}{
					\draw[black,thick] (\i,\j) -- (\i+1,\j);
				}
			}
			\foreach \i in {0,...,10}{
				\draw[black,thick] (\i,0) -- (\i,1);
			}
			\draw[black,thick] (6,0) -- (6.25,0);
			\draw[dark green, line width = 4pt] (6,1) -- (6.25,1);
			\draw[black,thick] (6.75,0) -- (7,0);
			\draw[dark green, line width = 4pt] (6.75,1) -- (7,1);
			\draw[black,thick] (-.4,0) -- (0,0);
			\draw[black,thick] (-.4,1) -- (0,1);
			\draw[black,thick] (10,1) -- (10.4,1);
			\draw[black,thick] (10,0) -- (10.4,0);
			
			\foreach \i in {1,...,5,7,8,9}{
				\draw[dark green, line width = 4pt](\i,1) -- (\i+1,1);
			}
			\draw[dark green, line width = 4pt](0,0) -- (1,0);
			\draw[dark green, line width = 4pt](1,0) -- (1,1);
			\foreach \i in {0,1,2}{
				\foreach \j in {0,1}{
					\draw  node[fill,circle,inner sep=0pt,minimum size=2pt] at (6.36 + .14*\i,\j){};
					\draw  node[fill,circle,inner sep=0pt,minimum size=2pt] at (10.6 + .14*\i,\j){};
					\draw  node[fill,circle,inner sep=0pt,minimum size=2pt] at (-.6 - .14*\i,\j){};
				}
			}
			\foreach \i in {3,5,7,9}{
				\draw[dark green, line width = 4pt](\i,0) -- (\i,1);}
			
			\foreach \i in {0,...,10}{
				\node[hollow node] (\i0) at (\i,0){};
				\node[hollow node] (\i1) at (\i,1){};
			}
			\foreach \i in {0,...,6}{
				\node at (\i,-.4){$v_{\i}$};
				\node at (\i,1.4){$u_{\i}$};
			}
			\foreach \i in {1,2,3}{
				\node at (10-\i,-.4){$v_{2k-\i}$};
				\node at (10-\i,1.4){$u_{2k-\i}$};
			}
			\node at (10,-.4){$v_{2k}$};
			\node at (10,1.4){$u_{2k}$};
			\node[red node] at (01){};
			\foreach \i in {2,4,6,8,10}{
				\node[red node] at (\i0){};
			}
			
		\end{tikzpicture}
		\caption{Minimal Diameter Infecting Set, I}
		\label{fig:narrowest odd tree}
		
	\end{figure}
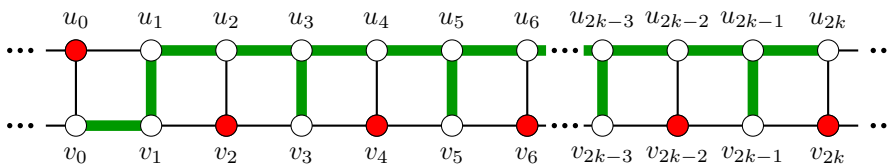
	
	Larger diameters may be obtained by ``weaving in and out.'' For example, exchanging $v_4$ for $u_4$ in Figure \ref{fig:narrowest odd tree} to ``weave around'' $u_4$ gives a diameter of $n+2$ (see Figure \ref{fig:middle odd tree}). Additional bumps may be added and each increases the diameter of $T$ by $2$, to a maximum diameter of $n+2\lfloor\frac{n-3}{4}\rfloor$ (see Figure \ref{fig:widest odd tree} for $k$ odd).
	
	\begin{figure}[H]
		\centering
		\begin{tikzpicture}
			\tikzstyle{hollow node}=[draw,circle, fill=white, outer sep=-4pt,inner sep=0pt,
			minimum width=8pt, above]
			\tikzstyle{red node}=[draw,circle, fill=red, outer sep=-4pt,inner sep=0pt,
			minimum width=8pt, above]
			\foreach \i in {0,...,5,7,8,9}{
				\foreach \j in {0,1}{
					\draw[black,thick] (\i,\j) -- (\i+1,\j);
				}
			}
			\foreach \i in {0,...,10}{
				\draw[black,thick] (\i,0) -- (\i,1);
			}
			\draw[black,thick] (6,0) -- (6.25,0);
			\draw[dark green, line width = 4pt] (6,1) -- (6.25,1);
			\draw[black,thick] (6.75,0) -- (7,0);
			\draw[dark green, line width = 4pt] (6.75,1) -- (7,1);
			\draw[black,thick] (-.4,0) -- (0,0);
			\draw[black,thick] (-.4,1) -- (0,1);
			\draw[black,thick] (10,1) -- (10.4,1);
			\draw[black,thick] (10,0) -- (10.4,0);
			
			\foreach \i in {1,2,5,7,8,9}{
				\draw[dark green, line width = 4pt](\i,1) -- (\i+1,1);
			}
			\foreach \i in {0,3,4}{
				\draw[dark green, line width = 4pt](\i,0) -- (\i+1,0);}
			
			\foreach \i in {0,1,2}{
				\foreach \j in {0,1}{
					\draw  node[fill,circle,inner sep=0pt,minimum size=2pt] at (6.36 + .14*\i,\j){};
					\draw  node[fill,circle,inner sep=0pt,minimum size=2pt] at (10.6 + .14*\i,\j){};
					\draw  node[fill,circle,inner sep=0pt,minimum size=2pt] at (-.6 - .14*\i,\j){};
				}
			}
			\foreach \i in {1,3,5,7,9}{
				\draw[dark green, line width = 4pt](\i,0) -- (\i,1);}
			
			\foreach \i in {0,...,10}{
				\node[hollow node] (\i0) at (\i,0){};
				\node[hollow node] (\i1) at (\i,1){};
			}
			\foreach \i in {0,...,6}{
				\node at (\i,-.4){$v_{\i}$};
				\node at (\i,1.4){$u_{\i}$};
			}
			\foreach \i in {1,2,3}{
				\node at (10-\i,-.4){$v_{2k-\i}$};
				\node at (10-\i,1.4){$u_{2k-\i}$};
			}
			\node at (10,-.4){$v_{2k}$};
			\node at (10,1.4){$u_{2k}$};
			\node[red node] at (01){};
			\foreach \i in {2,6,8,10}{
				\node[red node] at (\i0){};
			}
			\node[red node] at (41){};
			
		\end{tikzpicture}
		\caption{Increasing the Diameter by $2$}
		\label{fig:middle odd tree}
		
	\end{figure}
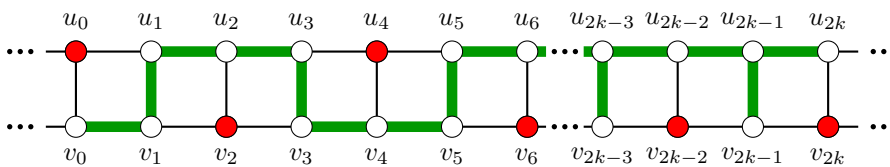

	\begin{figure}[H]
		\centering
		\begin{tikzpicture}
			\tikzstyle{hollow node}=[draw,circle, fill=white, outer sep=-4pt,inner sep=0pt,
			minimum width=8pt, above]
			\tikzstyle{red node}=[draw,circle, fill=red, outer sep=-4pt,inner sep=0pt,
			minimum width=8pt, above]
			\foreach \i in {0,...,5,7,8,9}{
				\foreach \j in {0,1}{
					\draw[black,thick] (\i,\j) -- (\i+1,\j);
				}
			}
			\foreach \i in {0,...,10}{
				\draw[black,thick] (\i,0) -- (\i,1);
			}
			\draw[black,thick] (6,0) -- (6.25,0);
			\draw[dark green, line width = 4pt] (6,1) -- (6.25,1);
			\draw[black,thick] (6.75,0) -- (7,0);
			\draw[dark green, line width = 4pt] (6.75,1) -- (7,1);
			\draw[black,thick] (-.4,0) -- (0,0);
			\draw[black,thick] (-.4,1) -- (0,1);
			\draw[black,thick] (10,1) -- (10.4,1);
			\draw[black,thick] (10,0) -- (10.4,0);
			
			\foreach \i in {1,2,5,9}{
				\draw[dark green, line width = 4pt](\i,1) -- (\i+1,1);
			}
			\foreach \i in {0,3,4,7,8}{
				\draw[dark green, line width = 4pt](\i,0) -- (\i+1,0);}
			
			\foreach \i in {0,1,2}{
				\foreach \j in {0,1}{
					\draw  node[fill,circle,inner sep=0pt,minimum size=2pt] at (6.36 + .14*\i,\j){};
					\draw  node[fill,circle,inner sep=0pt,minimum size=2pt] at (10.6 + .14*\i,\j){};
					\draw  node[fill,circle,inner sep=0pt,minimum size=2pt] at (-.6 - .14*\i,\j){};
				}
			}
			\foreach \i in {1,3,5,7,9}{
				\draw[dark green, line width = 4pt](\i,0) -- (\i,1);}
			
			\foreach \i in {0,...,10}{
				\node[hollow node] (\i0) at (\i,0){};
				\node[hollow node] (\i1) at (\i,1){};
			}
			\foreach \i in {0,...,6}{
				\node at (\i,-.4){$v_{\i}$};
				\node at (\i,1.4){$u_{\i}$};
			}
			\foreach \i in {1,2,3}{
				\node at (10-\i,-.4){$v_{2k-\i}$};
				\node at (10-\i,1.4){$u_{2k-\i}$};
			}
			\node at (10,-.4){$v_{2k}$};
			\node at (10,1.4){$u_{2k}$};
			\node[red node] at (01){};
			\foreach \i in {2,6,10}{
				\node[red node] at (\i0){};
			}
			\node[red node] at (41){};
			\node[red node] at (81){};

		\end{tikzpicture}
		\caption{Maximal Diameter Infecting Set, $k$ Odd}
		\label{fig:widest odd tree}
		
	\end{figure}
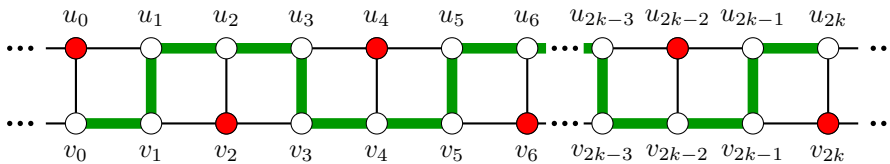

	Turn now to the case of $n$ even and write $n = 2k$ $(n\geq4)$. Again a minimum size infecting set $S$ has $k+1$ vertices. This case has several subcases, but as each is handled similarly to the previous case, we only sketch the major steps.
	
	By Theorem \ref{components thm}, $G[V-S]$ may be a forest with one or two connected components. Consider first the case where it is a single tree. In this case there now exists precisely one pair of neighboring vertices in $S$. We can still find a square face $F$ of the annulus with both horizontal edges not in $T$. As before, split the annulus at this face. There will be two subcases to examine.	
	
	The first subcase is when $F$ has the same form as Figure \ref{fig:1}.  The minimal diameter tree will look similar to Figure \ref{fig:narrowest odd tree}, except that $T$ will connect $v_1$ to $v_2$ before moving up to $u_2$ and continuing horizontally, and the indices will have a maximal label of $2k-1$. Weaving will increase the diameter by $2$ and the possible diameters of $T$ are $n, n+2, \ldots, n+ 2\lfloor\frac{n-4}{4}\rfloor$.

	The second subcase is when the neighboring vertices of $F$ are, after relabeling, $u_0$ and $v_0$. In this case, the tree $T$ with minimal diameter is pictured in Figure
	\ref{fig:narrowest even tree 2}. Weaving allows incrementing the diameter by $1$ and the possible diameters of $T$ are $n, n+1, \ldots, \frac 32 n -2$.
	
%	Case 2A(ii): $F$ has adjacent vertices from $S$ \textcolor{red}{(might be better to split (i/ii) by whether the edge of adjacent infected is vertical or horizontal?)}. Then without loss of generality the leftmost vertices are in $F\cap S$ and the rightmost vertices are in $F-S$. The remaining vertices of $S$ must again occupy alternating vertical edges, and again we can choose to weave or agree with the preceding choice to get diameters from $n$ to $n+2\lfloor\frac{n-4}{4}\rfloor$.
%	
	
	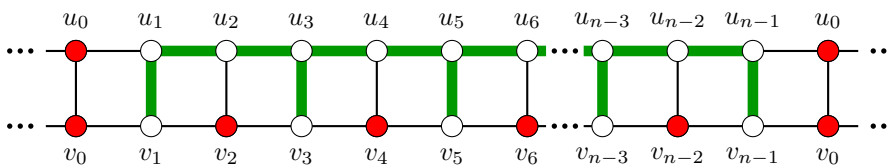
\begin{figure}[H]
		\centering
		\begin{tikzpicture}
			\tikzstyle{hollow node}=[draw,circle, fill=white, outer sep=-4pt,inner sep=0pt,
			minimum width=8pt, above]
			\tikzstyle{red node}=[draw,circle, fill=red, outer sep=-4pt,inner sep=0pt,
			minimum width=8pt, above]
			\foreach \i in {0,...,5,7,8,9}{
				\foreach \j in {0,1}{
					\draw[black,thick] (\i,\j) -- (\i+1,\j);
				}
			}
			\foreach \i in {0,...,10}{
				\draw[black,thick] (\i,0) -- (\i,1);
			}
			\draw[black,thick] (6,0) -- (6.25,0);
			\draw[dark green, line width = 4pt] (6,1) -- (6.25,1);
			\draw[black,thick] (6.75,0) -- (7,0);
			\draw[dark green, line width = 4pt] (6.75,1) -- (7,1);
			\draw[black,thick] (-.4,0) -- (0,0);
			\draw[black,thick] (-.4,1) -- (0,1);
			\draw[black,thick] (10,1) -- (10.4,1);
			\draw[black,thick] (10,0) -- (10.4,0);
			
			\foreach \i in {0,1,2}{
				\foreach \j in {0,1}{
					\draw  node[fill,circle,inner sep=0pt,minimum size=2pt] at (6.36 + .14*\i,\j){};
					\draw  node[fill,circle,inner sep=0pt,minimum size=2pt] at (10.6 + .14*\i,\j){};
					\draw  node[fill,circle,inner sep=0pt,minimum size=2pt] at (-.6 - .14*\i,\j){};
				}
			}
			
			\foreach \i in {1,...,5,7,8}{
				\draw[dark green, line width = 4pt](\i,1) -- (\i+1,1);
			}
			
			\foreach \i in {1,3,5,7,9}{
				\draw[dark green, line width = 4pt](\i,0) -- (\i,1);
			}

			\foreach \i in {0,...,10}{
				\node[hollow node] (\i0) at (\i,0){};
				\node[hollow node] (\i1) at (\i,1){};
			}
			\foreach \i in {0,...,6}{
				\node at (\i,-.4){$v_{\i}$};
				\node at (\i,1.4){$u_{\i}$};
			}
			\foreach \i in {1,2,3}{
				\node at (10-\i,-.4){$v_{n-\i}$};
				\node at (10-\i,1.4){$u_{n-\i}$};
			}
				\node at (10,-.4){$v_{0}$};
				\node at (10,1.4){$u_{0}$};
			
			\foreach \i in {2,4,6,8}{
				\node[red node] at (\i,0){};
			}
			\foreach \i in {0,1}{
				\node[red node] at (0,\i){};
				\node[red node] at (10,\i){};
			}
			
		\end{tikzpicture}
		\caption{Minimal Diameter Infecting Set, II}
		\label{fig:narrowest even tree 2}
	\end{figure}
	
	The final case is when $G[V-S]$ constitutes $2$ connected components and $S$ has no neighboring vertices. The infecting time is controlled by the diameter of the larger tree. The analysis is similar, but now the annulus may be cut in two places. After relabeling, there are two possibilities. The first is that $S$ contains $\{u_0,u_m, v_{m+1}, v_{2k-1}\}$ and the first tree connects $v_0$ to $v_m$ and the second tree connects $u_{m+1}$ to $u_{2k-1}$. The other is that $S$ contains $\{u_0,v_m, u_{m+1}, v_{2k-1}\}$ and the first tree connects $v_0$ to $u_m$ and the second tree connects $v_{m+1}$ to $u_{2k-1}$. A direct inspection shows that the smallest possible diameter for the larger tree is $\lfloor \frac{n}{2} \rfloor$. Weaving allows incrementing the diameter by $1$ and the possible diameters for the larger tree are $\lfloor \frac{n}{2} \rfloor, \lfloor \frac{n}{2} \rfloor +1, \ldots, n-2$.
	
	We now summarize the possible maximal diameters of connected components for the complement of a minimum size infecting set $S$ for the case of $n$ even. The diameter $d$ may be achieved for each $d$ satisfying
	$\lfloor \frac{n}{2} \rfloor \leq d \leq \frac 32 n -2$ with $d \not= n-1$.	
\end{proof}

\section{Further Results}

In this section, we examine some closely related variants of $G(n,k)$.

When working with two copies of $G(n,k)$, we will write $\{u_k',v_k':0\leq k \leq n-1\}$ for the second set of vertices. We then write $G(n,k)\#G(n,k)$ for the graph obtained by surgery on the two copies of $G(n,k)$ along the edges $u_0 u_1$ and $u_0' u_1'$. By this we mean that $G(n,k)\#G(n,k)$ is the union of the two graphs with the edges $u_0 u_1$ and $u_0' u_1'$ replaced by $u_0 u_0'$ and $u_1 u_1'$. The following result calculates the $P_3$-hull number of $G(n,k)\#G(n,k)$. In fact, the result generalizes to most types of surgery along other edges.

\begin{theorem}\label{thm_surgery}
 $h_{P_3}(G(n,k)\#G(n,k))=n+1$.
\end{theorem}

\begin{proof}
	Write $G=G(n,k)\# G(n,k)$. Then $G$ is cubic with $4n$ vertices and $6n$ edges, so that, by Theorem \ref{Main Theorem} and
	\cite[Corollary~2]{Betti-Deficiancy-Thing}, its $P_3$-hull number is at least $n+1$. We distinguish two cases. For each recall that $G(n,k)$ has a minimum size infecting set of size $\lceil\frac{n+1}{2}\rceil$ with an explicit realization in Corollary \ref{infecting_set}. We will use notation from that Corollary.
	
Consider first the case where $n$ is odd so that $l$ (and $c$) are also odd. Apply an index shift to $G(n,k)\# G(n,k)$, so that the surgery happens along the edges $u_{c-1} u_c$ and $u_{c-1}' u_c'$. By inspection, the union of our infecting sets for each $G(n,k)$ infects $G$. As the order of this union is $n+1$, we are done.

Consider next the case where $n$ is even. When $l$ is also even, a slight modification of $\mathcal S$ provides the infecting set
\[ \{u_{0}\} \cup \{v_{j+ik}, v_{j+ik}':0\leq i\leq l-1,\, 0\leq j\leq c-1,\, i\text{ even }\} \]
of size $n+1$.
When $l$ is odd, then $c$ is even. Apply an index shift to $G(n,k)\# G(n,k)$, so that the surgery happens along the edges $u_{c-1} u_c$ and $u_{c-1}' u_c'$.
This time, an infecting set is given by the union of the two infecting sets for $G(n,k)$ minus $\{u'_{c-1}\}$. The size of this set is $(\frac{n}{2}+1) +\frac{n}{2} = n+1$, and we are done.
\end{proof}

The following is a generalization of $G(n,k)$.

\begin{definition}
	Let $\sigma$ be a nontrivial permutation in $S_n$ and write it as a product of disjoint, nontrivial cycles, $\sigma = \prod_{i=1}^m \sigma_i$. Assume $|\sigma_i| \geq 3$ for $1 \leq i \leq m$.
	
	The \textit{GGP graph} for $\sigma$, $GGP(n,\sigma)$ is a graph on $2n$ vertices with vertex set
	$$V=\{u_0, u_1,\dots, u_{n-1}, v_0, v_1, \dots, v_{n-1}\}$$
	and edge set
	\[ E= \{u_i u_{i+1}, \, u_i v_i, \, v_i v_{\sigma(i)}: 0\leq i \leq n-1\} \]
	with indices interpreted modulo $n$.
\end{definition}

The assumption above that each $|\sigma_i| \geq 3$ guarantees that $GGP(n,\sigma)$ is a simple cubic graph. Note that if $\sigma(i)=i+k$, then $GGP(n,\sigma) = G(n,k)$.

\newcount\CircNum
\newcommand\Clr{black}
\begin{example} \label{3 inner triangles}
	The following is an example of a GGP graph.
	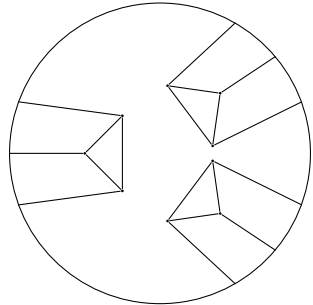
\begin{figure}[h]
		\centering
		\begin{tikzpicture}
			% draw the main circle
			\node [circle, draw, minimum size=4cm] (c) {};
			
			\foreach \a in {1,2,...,18}{
				\ifnum\a=\the\CircNum
				\renewcommand\Clr{blue}
				\else
				\renewcommand\Clr{black}
				\fi
				\node[circle,inner sep=0pt,minimum size=1pt] at (c.\a*360/18) {}
				{};
			};
			
			\newlength\Radius
			\setlength\Radius{2cm}
			\node[fill,circle,inner sep=0pt,minimum size=1pt](A1) at (-1,0) {};
			\node[fill,circle,inner sep=0pt,minimum size=1pt](A2) at (-0.5,-0.5) {};
			\node[fill,circle,inner sep=0pt,minimum size=1pt](A3) at (-0.5,0.5) {};
			
			\draw (A1) -- (A2) -- (A3) --(A1);
			
			\draw (c.9*360/18) -- (A1);
			\draw (c.10*360/18) -- (A2);
			\draw (c.8*360/18) -- (A3);
			
			\node[fill,circle,inner sep=0pt,minimum size=1pt](A1) at (0.8,0.8) {};
			\node[fill,circle,inner sep=0pt,minimum size=1pt](A2) at (0.7,0.1) {};
			\node[fill,circle,inner sep=0pt,minimum size=1pt](A3) at (0.1,0.9) {};
			
			\draw (A1) -- (A2) -- (A3) --(A1);
			
			\draw (c.2*360/18) -- (A1);
			\draw (c.1*360/18) -- (A2);
			\draw (c.3*360/18) -- (A3);
			
			\node[fill,circle,inner sep=0pt,minimum size=1pt](A1) at (0.8,-0.8) {};
			\node[fill,circle,inner sep=0pt,minimum size=1pt](A2) at (0.7,-0.1) {};
			\node[fill,circle,inner sep=0pt,minimum size=1pt](A3) at (0.1,-0.9) {};
			
			\draw (A1) -- (A2) -- (A3) --(A1);
			
			\draw (c.16*360/18) -- (A1);
			\draw (c.17*360/18) -- (A2);
			\draw (c.15*360/18) -- (A3);
		\end{tikzpicture}
		\caption{Example of $GGP(9,\sigma)$}
		\label{fig:GGP1}
	\end{figure}
\end{example}

	Note that the example of $GGP(9,\sigma)$ in Figure \ref{fig:GGP1} has a decycling number of 6, but that $\lceil\frac{9+1}{2}\rceil = 5$. It follows that Corollary \ref{p3_hull_GP} does not hold for all $GGP$ graphs. It is expected that a general formula for the $P_3$-hull number for all $GGP$ graphs is difficult. We can, however, provide some bounds.

\begin{theorem}\label{odd cycles upper bound}
	Let $\sigma$ be a nontrivial permutation in $S_n$ written as a product of disjoint, nontrivial cycles, $\sigma = \prod_{i=1}^m \sigma_i$, with $n_i = |\sigma_i| \geq 3$ for $1 \leq i \leq m$. Write $k$ for the number of $n_i$ that are odd and let $G=GGP(n,\sigma)$.
	
	Then \[h_{P_3}(G) \geq \bigg\lceil \frac{n+1}{2} \bigg\rceil.\] If each $n_i$ is even, then  \[h_{P_3}(G)  = \frac{n+2}{2}.\] If there are odd $n_i$, then \[h_{P_3}(G)  \leq \frac{n+k}{2}.\]

Note, in particular, $h_{P_3}(G)  = \frac{n+2}{2}$ for $k=0$, $h_{P_3}(G)  = \frac{n+1}{2}$ for $k=1$, and $h_{P_3}(G)  = \frac{n+2}{2}$ for $k=2$.
\end{theorem}

\begin{proof}
    Since $G=(V,E)$ is cubic, Theorem \ref{Main Theorem} and
    \cite[Corollary 2]{Betti-Deficiancy-Thing} provide the lower bound.
    In addition, \cite[Theorem 3]{Betti-Deficiancy-Thing} shows that
    $h_{P_3}(G) = \gamma_M(G) + \xi(G)$, where $\gamma_M(G)$ is the maximum genus of $G$ and $\xi(G)$ is its Betti deficiency. As $G$ is connected, $\gamma_M(G) = \frac{\beta(G) - \xi(G)}{2}$, where $\beta(G)$ is the cycle rank of $G$, which is $|E|-|V|+1 = n + 1$ for a cubic connected graph on $2n$ vertices. Thus $h_{P_3}(G) = \frac{\beta(G) + \xi(G)}{2}=\frac{n+1+\xi(G)}{2}$.

    Consider first the case of $k = 0$. Recall that the Betti deficiency, $\xi(G)$, is the minimum over all co-trees of $G$ of the number of components with an odd number of edges. So, we only need to construct a co-tree of $G$ with at most $1$ odd component. We can construct such a tree $T$ by taking all exterior edges $u_iu_{i+1}$ except $u_0u_1$, plus all in-between edges, $u_i v_i$. Thus the components of the co-tree corresponding to $T$ are the interior cycles and $u_0u_1$. There are no odd interior cycles so $u_0u_1$ is the only odd component. Thus $\xi(G)\leq 1$ and $h_{P_3}(G)  \le \frac{n+2}{2}$. However, note that $n =\sum_{i=1}^m n_i\equiv k \mbox{ mod } 2$ is even and $h_{P_3}(G) \geq \lceil \frac{n+1}{2} \rceil = \frac{n+2}{2}$ by our lower bound.

    Turn now to the case of $k \geq 1$. To show that $\xi(G)\leq k-1$, it suffices to construct a co-tree, $T$, of $G$ with $k-1$ odd components. As a visual guide for the construction, see Figure \ref{fig:my_label}.

    \begin{figure}[H]\label{co-tree picture}
    	\centering
    	\begin{tikzpicture}
    		\node[fill,circle,inner sep=0pt,minimum size=1pt] (a) at (4*1,4*0) {};
    		\node[fill,circle,inner sep=0pt,minimum size=1pt] (b) at (4*.866,.5*4) {};
    		\node[fill,circle,inner sep=0pt,minimum size=1pt] (c) at (4*.5,.866*4) {};
    		\node[fill,circle,inner sep=0pt,minimum size=1pt, label = \hspace{0em}$u^*$] (d) at (4*0,1*4) {};
    		\node[fill,circle,inner sep=0pt,minimum size=1pt] (g) at (4*-1,0*4) {};
    		\node[fill,circle,inner sep=0pt,minimum size=1pt] (f) at (4*-.866,.5*4) {};
    		\node[fill,circle,inner sep=0pt,minimum size=1pt, label = $x^*$] (e) at (4*-.5,.866*4) {};
    		\node[fill,circle,inner sep=0pt,minimum size=1pt] (j) at (4*0,-1*4) {};
    		\node[fill,circle,inner sep=0pt,minimum size=1pt] (h) at (4*-.866,-.5*4) {};
    		\node[fill,circle,inner sep=0pt,minimum size=1pt] (i) at (4*-.5,-.866*4) {};
    		\node[fill,circle,inner sep=0pt,minimum size=1pt] (l) at (4*.866,-.5*4) {};
    		\node[fill,circle,inner sep=0pt,minimum size=1pt] (k) at (4*.5,-.866*4) {};
    		
    		\draw[black, ultra thick] (a) -- (b);
    		\draw[black, ultra thick] (b) -- (c);
    		\draw[black, ultra thick] (c) -- (d);
    		\draw[black, dashed] (d) -- (e);
    		\draw[black, ultra thick] (e) -- (f);
    		\draw[black, ultra thick] (f) -- (g);
    		\draw[black, ultra thick] (g) -- (h);
    		\draw[black, ultra thick] (h) -- (i);
    		\draw[black, ultra thick] (i) -- (j);
    		\draw[black, ultra thick] (j) -- (k);
    		\draw[black, ultra thick] (k) -- (l);
    		\draw[black, ultra thick] (l) -- (a);
    		
    		\node[fill,circle,inner sep=0pt,minimum size=1pt] (m) at (2.5*1,2.5*0) {};
    		\node[fill,circle,inner sep=0pt,minimum size=1pt] (n) at (2.5*.866,.5*2.5) {};
    		\node[fill,circle,inner sep=0pt,minimum size=1pt] (o) at (2.5*.5,.866*2.5) {};
    		\node[fill,circle,inner sep=0pt,minimum size=1pt, label = \hspace{1.5em}$v^*$] (p) at (2.5*0,1*2.5) {};
    		\node[fill,circle,inner sep=0pt,minimum size=1pt] (s) at (2.5*-1,0*2.5) {};
    		\node[fill,circle,inner sep=0pt,minimum size=1pt] (r) at (2.5*-.866,.5*2.5) {};
    		\node[fill,circle,inner sep=0pt,minimum size=1pt] (q) at (2.5*-.5,.866*2.5) {};
    		\node[fill,circle,inner sep=0pt,minimum size=1pt] (v) at (2.5*0,-1*2.5) {};
    		\node[fill,circle,inner sep=0pt,minimum size=1pt] (t) at (2.5*-.866,-.5*2.5) {};
    		\node[fill,circle,inner sep=0pt,minimum size=1pt] (u) at (2.5*-.5,-.866*2.5) {};
    		\node[fill,circle,inner sep=0pt,minimum size=1pt] (x) at (2.5*.866,-.5*2.5) {};
    		\node[fill,circle,inner sep=0pt,minimum size=1pt] (w) at (2.5*.5,-.866*2.5) {};
    		
    		\draw[black, ultra thick] (p) -- (q);
    		\draw[black, dashed] (q) -- (r);
    		\draw[black, dashed] (r) -- (n);
    		\draw[black, dashed] (n) -- (o);
    		\draw[black, dashed] (o) -- (p);
    		
    		\draw[black, dashed] (t) -- (u);
    		\draw[black, dashed] (u) -- (x);
    		\draw[black, dashed] (x) -- (w);
    		\draw[black, dashed] (w) -- (t);
    		\draw[black, dashed] (s) -- (v);
    		\draw[black, dashed] (v) -- (m);
    		\draw[black, dashed] (m) -- (s);
    		
    		\draw[black, ultra thick] (m) -- (a);
    		\draw[black, ultra thick] (n) -- (b);
    		\draw[black, ultra thick] (o) -- (c);
    		\draw[black, dashed] (p) -- (d);
    		\draw[black, ultra thick] (q) -- (e);
    		\draw[black, ultra thick] (r) -- (f);
    		\draw[black, ultra thick] (s) -- (g);
    		\draw[black, ultra thick] (t) -- (h);
    		\draw[black, ultra thick] (u) -- (i);
    		\draw[black, ultra thick] (v) -- (j);
    		\draw[black, ultra thick] (w) -- (k);
    		\draw[black, ultra thick] (x) -- (l);
    		
    	\end{tikzpicture}
    	\centering
    	\caption{A $GGP$ on 24 Vertices With 2 Odd Inner Cycles}
    	\label{fig:my_label}
    \end{figure}
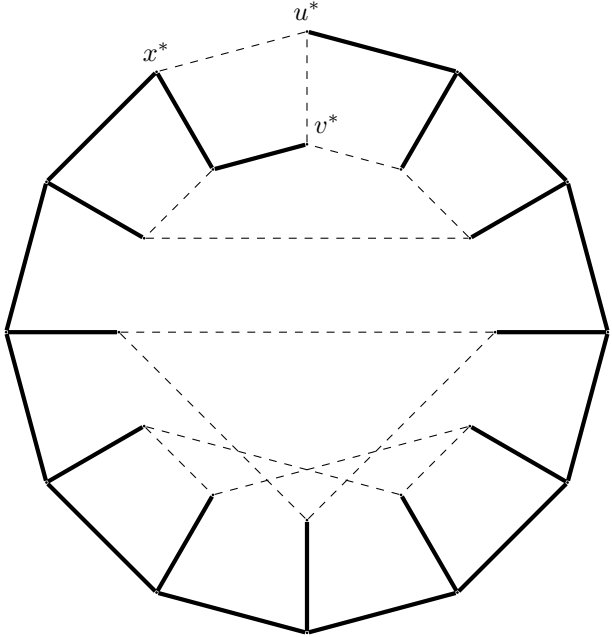

    Begin with any interior cycle, $C$, of odd length in $G$. Pick an edge $e^* = u^*v^*$ between $v^*\in C$ and $u^*\notin C$. Include all edges $u_i v_i$, except $e^*$, in $T$. Label the neighbors of $u^*$ on the exterior cycle as $x^*$ and $y^*$. Also include every exterior edge, $u_iu_{i+1}$ except $u^*x^*$, in $T$ (this includes $u^*y^*$). Finally, add to $T$ an edge of $C$ that is incident to $v^*$. This results in a spanning tree of~$G$.

    Write $H$ for the co-tree of $G$ corresponding to $T$. Observe that the remaining edges of $C$, together with $e^*$ and $u^*x^*$, form an even component of $H$. The only other components of $H$ are the rest of the interior cycles. Thus the number of odd components of $H$ is one less than the number of odd interior cycles of $G$, so $\xi(G)\leq k-1$.
\end{proof}

The upper bound in Theorem \ref{odd cycles upper bound} is achieved for certain combinations of $n$ and $k$. For instance, Figure \ref{fig:GGP1} has $n = 9, k = 3$ and $h_{P_3}(G) = 6$.
More generally, $n=3k$ and $h_{P_3}(G) = 2k$ holds for the corresponding graph with $k$ triangles in the interior achieving the upper bound.
However, there are quite general cases where the lower bound is achieved. The next theorem gives a sufficient condition based on an interleaving principle.

\begin{theorem}\label{thm:minGGP}
	Let $\sigma$ be a nontrivial permutation in $S_n$ written as a product of disjoint, nontrivial cycles, $\sigma = \prod_{i=1}^m \sigma_i$, with $n_i = |\sigma_i| \geq 3$ for $1 \leq i \leq m$.
	Let $G=GGP(n,\sigma)$. If there are $k$ odd $n_i$, assume there is a path of $k$ exterior vertices $u_j$, each of which connects to a different odd internal cycle.
Then $h_{P_3} (G) = \lceil\frac{n+1}{2}\rceil$.
\end{theorem}

\begin{proof}
	Define an infecting set by following a similar strategy to Corollary \ref{infecting_set}. Relabel so that the vertices of our path of $k$ exterior vertices are $u_0, \ldots, u_{k-1}$. Each odd internal cycle connects uniquely to this path by some edge $u_i v_i$, $0\le i\le k-1$.
	
	To define the initial infecting set, $\mathcal S$, for each internal cycle, $C_i$, add to $\mathcal S$ every other vertex of the cycle as was done with $\mathcal S_v$, avoiding the points $v_0, \ldots, v_{k-1}$ on the odd internal cycles. If there are only even cycles, add one external vertex for a total of $\frac n2+1 =\lceil\frac{n+1}{2}\rceil$ points. If there are odd cycles, add vertices from our path of $k$ exterior vertices $u_0, \ldots, u_{k-1}$ as was done with $\mathcal S_u$ in Corollary \ref{infecting_set} in the case of $l$ odd. It is straightforward to verify this is an infecting set of size $\lceil\frac{n+1}{2}\rceil$.
\end{proof}

There are numerous generalizations of Theorem \ref{thm:minGGP} that can be made and it would be very interesting to determine exactly when the possible minimal and maximal $P_3$-hull numbers are obtained for $GGP$ graphs.

%%%%%%%%%%%%%%%%%%%%%%%%%%%%%%%%%%%%%%%%%%%%%%%%%%%%%%%%%%%%%%
%%%%%%%%%%%%%%%%%%%%%%%%%%%%%%%%%%%%%%%%%%%%%%%%%%%%%%%%%%%%%%

\bibliographystyle{plain}
\bibliography{p3Hull}

\section{Statements and Declarations}
The authors declare that no funds, grants, or other support were received during the preparation of this manuscript.

The authors have no relevant financial or non-financial interests to disclose.

No additional data is needed for this paper.

\end{document}